\documentclass[11pt]{amsart}

\newif\ifdraft\draftfalse

\usepackage{amssymb}
\usepackage{amsthm}
\usepackage{eucal}
\usepackage[english]{babel}
\usepackage{nicefrac}
\usepackage{times}
\usepackage{latexsym,amscd,amsmath,amsthm,amssymb}

\makeatletter
\def\@begintheorem#1#2[#3]{%
    \def\naam{#1}
  \deferred@thm@head{\the\thm@headfont \thm@indent
    \@ifempty{#1}{\let\thmname\@gobble}{\let\thmname\@iden}%
    \@ifempty{#2}{\let\thmnumber\@gobble}{\let\thmnumber\@iden}%
    \@ifempty{#3}{\let\thmnote\@gobble}{\let\thmnote\@iden}%
    \thm@swap\swappedhead\thmhead{#1}{#2}{#3}%
    \the\thm@headpunct
    \thmheadnl 
    \hskip\thm@headsep
  }%
  \ignorespaces}
\makeatother

\newcommand{\kantlijndraft}[1]{\ifdraft\hspace{-\lastskip}%
\vadjust{\vspace{-1mm}\smash{\llap{{\tt #1}\hspace{8mm}}}\vspace{1mm}}\fi}

\def\voegToe#1#2#3{\immediate\write1{\string\newlabel{#1}{{#2}{#3}}}}

\newcommand{\thlabel}[1]{\voegToe{#1}{\naam\noexpand~\thetheorem}{\thepage}\kantlijndraft{#1}}

\makeatletter
\renewcommand{\label}[1]{\voegToe{#1}{\@currentlabel}{\thepage}\kantlijndraft{#1}}
\makeatother

\newtheorem{theorem}{Theorem}[section]
\newtheorem{lemma}[theorem]{Lemma}
\newtheorem{corollary}[theorem]{Corollary}
\newtheorem{question}[theorem]{Question}

\newtheorem{proposition}[theorem]{Proposition}
\newtheorem*{mt}{Main Theorem}
\theoremstyle{definition}
\newtheorem{example}[theorem]{Example}
\newtheorem{definition}[theorem]{Definition}
\theoremstyle{remark}

\numberwithin{equation}{section}

\newtheorem{claim2}{\sc Claim}



\newcommand{\pichar}[1]{\ensuremath{\pi\chi(#1)}}

\newcommand{\sse}{\subseteq}						
\newcommand{\minus}{\backslash}						
\newcommand{\Un}{\bigcup}							
\newcommand{\un}{\cup}								
\newcommand{\Meet}{\bigcap}							
\newcommand{\meet}{\cap}							
\newcommand{\es}{\varnothing}						

\newcommand{\closure}[1]{\ensuremath{\overline{#1}}}

\newcommand{\scr}[1]{\ensuremath{\mathcal{#1}}}

\def\cprime{$'$}
\def\cont{\mathfrak{c}}
\def\sapirovskii{{\v{S}}apirovski{\u\i}}
\def\arhangelskii{Arhangel{\cprime}ski{\u\i}}

\def\juhasz{Juh{\'a}sz}

\begin{document}

\title{On the weak tightness, Hausdorff spaces, and power homogeneous compacta}

\author{Nathan Carlson}\address{Department of Mathematics, California Lutheran University, 60 W. Olsen Rd, MC 3750, Thousand Oaks, CA 91360 USA}
\email{ncarlson@callutheran.edu}



\begin{abstract} 
Motivated by results of \juhasz~and van Mill in \cite{JVM2017}, we define the cardinal invariant $wt(X)$, the weak tightness of a topological space $X$, and show that $|X|\leq 2^{L(X)wt(X)\psi(X)}$ for any Hausdorff space $X$ (Theorem~\ref{mt1}). As $wt(X)\leq t(X)$ for any space $X$, this generalizes the well-known cardinal inequality $|X|\leq 2^{L(X)t(X)\psi(X)}$ for Hausdorff spaces (\arhangelskii~\cite{Arh1969},~\sapirovskii~ \cite{Sap1972}) in a new direction. Theorem~\ref{mt1} is generalized further using covers by $G_\kappa$-sets, where $\kappa$ is a cardinal, to show that if $X$ is a power homogeneous compactum with a countable cover of dense, countably tight subspaces then $|X|\leq\cont$, the cardinality of the continuum. This extends a result in \cite{JVM2017} to the power homogeneous setting.
\end{abstract}

\maketitle

\section{Introduction.}

In 2006 R. de la Vega~\cite{DeLaVega2006} answered a long-standing question of A.V.~\arhangelskii\space by showing that the cardinality of any homogeneous compactum is at most $2^{t(X)}$, where the cardinal invariant $t(X)$ is the \emph{tightness} of $X$ (Definition~\ref{tightness}).  This was a landmark result in the theory of homogeneous topological spaces and has motivated subsequent work and generalizations. Recall that a space $X$ is \emph{homogeneous} if for all $x,y\in X$ there exists a homeomorphism $h:x\to y$ such that $h(x)=y$. Roughly, a space is homogeneous if all points in the space share identical topological properties. A space is \emph{power homogeneous} if there exists a cardinal $\kappa$ such that $X^\kappa$ is homogeneous. 

Recently, \juhasz~and van Mill \cite{JVM2017} introduced fundamentally new techniques and gave the following intricate improvement of De la Vega's theorem. Their result generalizes that theorem in the countably tight case. 

\begin{theorem}[\juhasz, van Mill~\cite{JVM2017}]\label{jvm}
If a compactum $X$ is the union of countably many dense countably tight subspaces and $X^\omega$ is homogeneous, then $|X|\leq\cont$. 
\end{theorem}

It was observed by \arhangelskii~in \cite{Arh2006} that De la Vega's theorem in fact follows from the following theorem of Pykeev involving the $G_\kappa$-\emph{modification} $X_\kappa$ of a space $X$, the space formed on the same underlying set as $X$ with topology generated by the $G_\kappa$-sets. (See Definition~\ref{gk}). 

\begin{theorem}[Pytkeev, \cite{Pyt1985}]\label{pytkeev}
If $X$ is compact and $\kappa$ a cardinal, then  $L(X_\kappa)\leq 2^{t(X)\cdot\kappa}$. 
\end{theorem}

A short, direct proof of De la Vega's theorem using Pytkeev's theorem was given in \cite{CR2012}. In addition, a generalization of Pytkeev's theorem proved in that paper has as a short corollary the well-known cardinality bound $2^{L(X)t(X)\psi(X)}$ for the cardinality of any Hausdorff space $X$ (\arhangelskii~\cite{Arh1969},~\sapirovskii~ \cite{Sap1972}). This established that the $G_\kappa$-modification of a space can be used to derive related cardinality bounds for both general spaces and homogeneous spaces simultaneously. These related bounds may be thought of as ``companion'' bounds. In the compact setting, De la Vega's theorem may be thought of as the homogeneous companion bound to the cardinality bound $2^{\chi(X)}$ for general compacta. Both follow from Pytkeev's theorem. Results in a similar vein occur in \cite{BC2017}, \cite{BS2017}, \cite{CPR2012}, \cite{SS2017}, and elsewhere.

In this study we develop the companion bound for general spaces to Theorem~\ref{jvm} above. To this end we introduce the \emph{weak tightness} $wt(X)$ of a space (Definition~\ref{weakTightness}), which satisfies $wt(X)\leq t(X)$, and prove the following theorem and central result in this paper. For a space $X$, the $G^c_\kappa$-modification $X^c_\kappa$ of $X$ is a variation of $X_\kappa$ (see Definition~\ref{gk}). Note that if $X$ is regular then $X^c_\kappa\approx X_\kappa$. 

\begin{mt}
For any space $X$ and infinite cardinal $\kappa$, $L(X^c_\kappa)\leq 2^{L(X)wt(X)\cdot\kappa}$.
\end{mt}

From the above theorem several corollaries follow. First, one corollary is that $|X|\leq 2^{L(X)wt(X)\psi(X)}$ for any Hausdorff space $X$ (Theorem~\ref{mt1}). This improves the~\arhangelskii-\sapirovskii~result. A second corollary is an extension of the \juhasz-van Mill result~\ref{jvm} into the power homogeneous setting. We show in Theorem~\ref{ph} that if $X$ is power homogeneous compactum that is the countable union of countably tight and dense in $X$ then $|X|\leq\cont$. This also generalizes the following theorem in~\cite{AVR2007}: if $X$ is a power homogeneous, countably tight compactum then $|X|\leq\cont$. 

Third, a result of Bella and Spadaro from~\cite{BS2017} follows as a corollary to the Main Theorem and can be strengthened to the power homogeneous setting. In that paper Theorem~\ref{jvm} was extended using the $G_\delta$-modification $X_\delta$ by showing that if $X$ is a homogeneous compactum which is the union of countably my dense countably tight subspaces, then $L(X_\delta)\leq\cont$. Using the Main Theorem, we show in Theorem~\ref{phbs} that this result holds even if the homogeneity property is replaced with the weaker power homogeneity property.

Finally, observe that it follows from the Main Theorem that if $X$ is compact and $\kappa$ a cardinal, then $L(X^c_\kappa)\leq 2^{wt(X)\cdot\kappa}$. This is a strengthening of Pytkeev's theorem.

This paper is organized as follows. In $\S2$ we define the cardinal invariant $wt(X)$ and give a direct proof that if $X$ is Hausdorff then $|X|\leq 2^{L(X)wt(X)\psi(X)}$ in Theorem~\ref{mt1}. This uses the fact that if $X$ is Hausdorff and $D\sse X$, then $|\closure{D}|\leq |D|^{wt(X)\psi_c(X)}$ (Theorem~\ref{cardClosure}). In $\S3$ we prove the full generalization of Theorem~\ref{mt1}, the Main Theorem mentioned above. Corollaries involving compact power homogeneous spaces follow in $\S4$.

For all undefined notions we refer the reader to~\cite{Engelking} and~\cite{Juhasz}. We make no implicit assumptions concerning separation axioms, although by \emph{compactum} we will mean a compact, Hausdorff space.

\section{An improved cardinality bound for Hausdorff spaces.}
In this section we give a direct proof of an improved cardinality bound for Hausdorff spaces. We begin with three definitions. First, we recall the definition of the tightness $t(X)$ of a space $X$. Observe any first countable space is countably tight.

\begin{definition}\label{tightness}
For $x\in X$, $t(x,X)$ is defined as the least infinite cardinal $\kappa$ such that for all $A\sse X$ if $x\in\closure{A}$ then there exists $B\sse A$ such that $|B|\leq\kappa$ and $x\in\closure{B}$. Then the \emph{tightness} $t(X)$ is defined by $t(X)=\sup\{t(x,X):x\in X\}$. A space is \emph{countably tight} (CT) if $t(X)$ is countable. A space is $\sigma$-CT if it is the countable union of CT subspaces.
\end{definition}

The following definition was defined in~\cite{CPR2012} and was used in~\cite{JVM2017}. For a subset of a set $X$ and a cardinal $\kappa$, we denote by $[A]^{\leq\kappa}$ the collection of subsets of $A$ of cardinality at most $\kappa$. Observe that $cl_\kappa(cl_\kappa A)=cl_\kappa A\sse\closure{A}$ and if $cl_\kappa A=X$ then $A$ is dense in $X$. 
\begin{definition} 
Given a cardinal $\kappa$, a space $X$, and $A\sse X$, the $\kappa$-\emph{closure} of $A$ is defined as $cl_\kappa A=\Un_{B\in[A]^{\leq\kappa}}\closure{B}$.
\end{definition}

We introduce the cardinal invariant $wt(X)$, motivated by results in~\cite{JVM2017}.
\begin{definition}\label{weakTightness}
Let $X$ be a space. The \emph{weak tightness} $wt(X)$ of $X$ is defined as the least infinite cardinal $\kappa$ for which there is a cover $\scr{C}$ of $X$ such that $|\scr{C}|\leq 2^\kappa$ and for all $C\in\scr{C}$, $t(C)\leq\kappa$ and $X=cl_{2^\kappa}C$. We say that $X$ is \emph{weakly countably tight} if $wt(X)=\omega$. 
\end{definition}
It is clear that $wt(X)\leq t(X)$. See Example~\ref{example} for an example of a space $X$ for which $wt(X)=\omega<\cont=t(X)$.

We work towards an improved bound for the cardinality of the closure of a subset of a Hausdorff space $X$. We will need the following straightforward proposition.

\begin{proposition}\label{prop}
Let $X$ be a space, $D\sse X$, and suppose there exists a cardinal $\kappa$ such that for all $x\in X$ there exists $\scr{B}_x\in\left[[D]^{\leq\kappa}\right]^{\leq\kappa}$ such that $\{x\}=\Meet_{B\in\scr{B}_x}\closure{B}$. Then $|X|\leq |D|^{\kappa}$.
\end{proposition}

\begin{proof}
Observe that if $\scr{B}_x=\scr{B}_y$ for $x,y\in X$, then $x=y$. This shows the function $\phi:X\to\left[[D]^{\leq\kappa}\right]^{\leq\kappa}$ defined by $\phi(x)=\scr{B}_x$ is one-to-one. Thus,  $|X|\leq\left|\left[[D]^{\leq\kappa}\right]^{\leq\kappa}\right|=\left(|D|^{\kappa}\right)^\kappa=|D|^{\kappa}$.
\end{proof}

The following theorem gives a bound for the cardinality of the closure of any subset $D$ of a Hausdorff space $X$. It represents a strengthening of the well-known cardinal inequality $|\closure{D}|\leq |D|^{t(X)\psi_c(X)}$.

\begin{theorem}\label{cardClosure}
If $X$ is Hausdorff and $D\sse X$, then $|\closure{D}|\leq |D|^{wt(X)\psi_c(X)}$.
\end{theorem}

\begin{proof}
Let $\kappa=wt(X)\psi_c(X)$ and let $\scr{C}$ be a cover of $X$ witnessing that $wt(X)\leq\kappa$. 

Fix $x\in\closure{D}$. As $\scr{C}$ is a cover of $X$, there exists $C\in\scr{C}$ such that $x\in X$. As $X=cl_{2^\kappa}C$, for each $d\in D$ there exists $C_d\in[C]^{\leq 2^{\kappa}}$ such that $d\in\closure{C_d}$. Set $A=\Un_{d\in D}C_d$ and observe that $A\sse C$, $D\sse\closure{A}$, and $|A|\leq |D|\cdot 2^{\kappa}\leq |D|^\kappa$. 

As $X$ is Hausdorff there exists an family $\scr{U}$ of open sets such that $|\scr{U}|\leq\psi_c(X)\leq\kappa$ and $\{x\}=\Meet\scr{U}=\Meet_{U\in\scr{U}}\closure{U}$. Let $U\in\scr{U}$ and let $V$ be an open set containing $x$. Then $x\in V\meet U$ and as $x\in\closure{D}\sse\closure{A}$ we have that $V\meet U\meet A\neq\es$. This shows $x\in\closure{U\meet A}$ and thus $x\in C\meet\closure{U\meet A}=cl_C(U\meet A)$ for all $U\in\scr{U}$. As $t(C)\leq\kappa$, for all $U$ there exists $B_U\in[U\meet A]^{\leq\kappa}$ such that $x\in cl_C(B_U)\sse\closure{B_U}$. Therefore, $x\in\Meet_{U\in\scr{U}}\closure{B_U}\sse\Meet_{U\in\scr{U}}\closure{U}=\{x\}$
and $\{x\}=\Meet_{U\in\scr{U}}\closure{B_U}$. As $\{B_U:U\in\scr{U}\}\in\left[[A]^{\leq\kappa}\right]^{\leq\kappa}$, by Proposition~\ref{prop} it follows that $|\closure{D}|\leq |A|^\kappa\leq\left(|A|^\kappa\right)^\kappa= |D|^\kappa$.
\end{proof}

In \cite{JVM2017}, \juhasz~and van Mill introduced the notion of a $\scr{C}$-saturated subset of a space $X$. 
\begin{definition}
Given a cover $\scr{C}$ of $X$, a subset $A\sse X$ is $\scr{C}$-\emph{saturated} if $A\meet C$ is dense in $A$ for every $C\in\scr{C}$. 
\end{definition}
It is clear that the union of $\scr{C}$-saturated subsets is $\scr{C}$-saturated. The following is a generalized version of Claim 1 in \cite{JVM2017}. 
\begin{lemma}\label{lemma2}
Let $X$ be a space, $\kappa$ a cardinal such that $wt(X)\leq\kappa$, and $\scr{C}$ be a cover of $X$ witnessing that $wt(X)\leq\kappa$. If $\scr{B}$ is an increasing chain of $\kappa^+$-many $\scr{C}$-saturated subsets of $X$, then 
$$\closure{\Un\scr{B}}=\Un_{B\in\scr{B}}\closure{B}.$$
\end{lemma}

\begin{proof}
Let $x\in\closure{\Un\scr{B}}$. There exists $C\in\scr{C}$ such that $x\in C$. As the union of $\scr{C}$-saturated subsets is $\scr{C}$-saturated we have that $\Un\scr{B}$ is $\scr{C}$-saturated. Therefore $x\in cl_C\left(C\meet\Un\scr{B}\right)$. As $t(C)\leq\kappa$ there is $A\sse C\meet\Un\scr{B}$ such that $x\in cl_CA\sse \closure{A}$ and $|A|\leq\kappa$. As $|\scr{B}|=\kappa^+$, $A\sse\Un\scr{B}$, and $\scr{B}$ is an increasing chain, there exists $B\in\scr{B}$ such that $A\sse B$. Thus $x\in\closure{B}$ which completes the proof.
\end{proof}

To prove the primary result in this section, Theorem~\ref{mt1} below, we proceed with a standard ``closing-off'' argument, also known as the ``Pol-\sapirovskii'' technique. These techniques are commonly used to establish bounds on the cardinality of a topological space.

\begin{theorem}\label{mt1}
If $X$ is Hausdorff then $|X|\leq 2^{L(X)wt(X)\psi(X)}$.
\end{theorem}

\begin{proof}
Let $\kappa=L(X)wt(X)\psi(X)$ and note that $\psi_c(X)\leq L(X)\psi(X)\leq\kappa$. For all $x\in X$ let $\scr{U}_x$ be a family of open sets such that $\{x\}=\Meet\scr{U}_x=\Meet_{U\in\scr{U}_x}\closure{U}$. If $B\sse X$, let $\scr{U}(B)=\Un_{x\in B}\scr{U}_x$. Let $\scr{C}$ be a cover of $X$ witnessing that $wt(X)\leq\kappa$. For every $x\in X$ we can fix a $\scr{C}$-saturated set $S(x)\in[X]^{\leq 2^\kappa}$ with $x\in S(x)$.

We build an increasing sequence $\{A_\alpha:\alpha<\kappa^+\}$ of subsets of $X$ such that 
\begin{enumerate}
\item $|A_\alpha|\leq 2^\kappa$.
\item $A_\alpha$ is $\scr{C}$-saturated for all $\alpha<\kappa^+$.
\item if $\scr{V}\in[\scr{U}(\closure{A_\alpha})]^{\leq\kappa}$ is such that $X\minus\Un\scr{V}\neq\es$, then $A_{\alpha+1}\minus\Un\scr{V}\neq\es$.
\end{enumerate}
Let $A_0$ be an any $\scr{C}$-saturated set such that $|A_0|\leq 2^\kappa$. For limit ordinals $\beta<\kappa^+$, let $A_\beta=\Un_{\alpha<\beta}A_\alpha$. Then $|A_\beta|\leq 2^\kappa$. As the union of $\scr{C}$-saturated sets is $\scr{C}$-saturated, as noted in \cite{JVM2017}, it follows that $A_\beta$ is $\scr{C}$-saturated. For successor ordinals $\beta+1$, for all $\scr{V}\in[\scr{U}(\closure{A_\beta})]^{\leq\kappa}$ with $X\minus\Un\scr{V}\neq\es$, let $x_\scr{V}\in X\minus\Un\scr{V}$. Define
$$A_{\beta+1}=A_\beta\un\Un\{S(x_\scr{V}):\scr{V}\in[\scr{U}(\closure{A_\beta})]^{\leq\kappa}\textup{ and }X\minus\Un\scr{V}\neq\es\}.$$
Note that $|\closure{A_\beta}|\leq |A_\beta|^{\leq\kappa}\leq 2^\kappa$ by Theorem~\ref{cardClosure} and thus $\left|[\scr{U}(\closure{A_\beta})]^{\leq\kappa}\right|\leq 2^\kappa$. As $|S(x_\scr{V})|\leq 2^\kappa$ we have $|A_{\beta+1}|\leq 2^\kappa$. As each $S(x_\scr{V})$ is $\scr{C}$-saturated and $A_\beta$ is $\scr{C}$-saturated we have that $A_{\beta+1}$ is $\scr{C}$-saturated. 

Let $A=\Un\{\closure{A_\alpha}:\alpha<\kappa^+\}$. By Lemma~\ref{lemma2} it follows that $A=\closure{\Un\{A_\alpha:\alpha<\kappa^+\}}$ and $A$ is closed. By Theorem~\ref{cardClosure}, $|A|\leq\left|\Un\{A_\alpha:\alpha<\kappa^+\}\right|^\kappa\leq\left(\kappa^+\cdot 2^\kappa\right)^\kappa=2^\kappa$. We show now that $X=A$. Suppose there exists $x\in X\minus A$. Then for all $a\in A$ there exists $U_a\in\scr{U}_a$ such that $x\notin\closure{U_a}$. Now $\scr{U}=\{U_a:a\in A\}$ is an open cover of  $A$. As $A$ is closed and $L(X)$ is hereditary on closed subsets, there exists $\scr{V}\in[\scr{U}]^{\leq\kappa}$ such that $A\sse\Un\scr{V}$. Then $\scr{V}\in[\scr{U}(\closure{A_\alpha})]^{\leq\kappa}$ for some $\alpha<\kappa^+$. Since $x\in X\minus\Un\scr{V}$, by (3) above $A_{\alpha+1}\minus\Un\scr{V}\neq\es$. Then
$$\es\neq A_{\alpha+1}\minus\Un\scr{V}\sse A\sse \Un\scr{V},$$
a contradition. Therefore $X=A$ and $|X|=|A|\leq 2^\kappa$.
\end{proof}

\begin{corollary}
If $X$ is a weakly countably tight, Lindel\"of Hausdorff space of countable pseudocharacter then $|X|\leq\cont$.
\end{corollary}

If there is a cover $\scr{C}$ of a space $X$ and a cardinal $\kappa$ for which $|\scr{C}|\leq 2^\kappa$ and for all $C\in\scr{C}$, $t(C)\leq\kappa$, and $X=cl_{2^\kappa}C$ then, by definition, $wt(X)\leq\kappa$. As Lemma~\ref{lemmaDense} shows, if there are additional cardinal restrictions on the $\pi$-character $\pi_\chi(X)$ or tightness $t(X)$ of a space, then the condition ``$X=cl_{2^\kappa}C$ for all $C\in\scr{C}$'' can be relaxed to ``$C$ is dense in $X$ for all $C\in\scr{C}$'' to still guarantee that $wt(X)\leq\kappa$.

\begin{lemma}\label{lemmaDense}
Let $X$ be a space, $\kappa$ a cardinal, and $\scr{C}$ a cover of $X$ such that $|\scr{C}|\leq 2^\kappa$, and for all $C\in\scr{C}$, $t(C)\leq\kappa$ and $C$ is dense in $X$. If $t(X)\leq 2^\kappa$ or $\pi_\chi(X)\leq 2^\kappa$ then $wt(X)\leq\kappa$.
\end{lemma}

\begin{proof}
We show that if either (a) $t(X)\leq 2^\kappa$ or (b) $\pi_\chi(X)\leq 2^\kappa$ then the cover $\scr{C}$ is a cover witnessing that $wt(X)\leq\kappa$. It only needs to be shown that $X=cl_{2^\kappa}(C)$ for all $C\in\scr{C}$. Let $C\in\scr{C}$ and $x\in X$. For (a), as $C$ is dense in $X$ and $t(X)\leq 2^\kappa$, then there exists $A\in[C]^{\leq 2^\kappa}$ such that $x\in\closure{A}$. Thus, $x\in cl_{2^\kappa}(C)$. For (b), let $\scr{B}$ be a local $\pi$-base at $x$ such that $|\scr{B}|\leq 2^\kappa$. As $C$ is dense in $X$, for all $B\in\scr{B}$ there exists $x_B\in B\meet C$. Then $x\in\closure{\{x_B:B\in\scr{B}\}}$ and $x\in cl_{2^\kappa}(C)$. In either case $X=cl_{2^\kappa}(C)$ as desired.
\end{proof}

Using Lemma~\ref{lemmaDense} and Theorem~\ref{mt1}, we have the following.
\begin{corollary}
Let $X$ be a Hausdorff space, $\kappa$ a cardinal, and $\scr{C}$ a cover of $X$ such that $|\scr{C}|\leq 2^\kappa$, and for all $C\in\scr{C}$, $t(C)\leq\kappa$ and $C$ is dense in $X$. If $t(X)\leq 2^\kappa$ or $\pi_\chi(X)\leq 2^\kappa$ then $|X|\leq 2^{L(X)\psi(X)\cdot\kappa}$.
\end{corollary}

We end this section with an example of a weakly countably tight space that is not countably tight. The example is a particular case of one given in~\cite{JVM2017}; our work here is to verify it is weakly countably tight. We do this using Lemma~\ref{lemmaDense}.

\begin{example}[$wt(X)=\omega<\cont=t(X)$]\label{example}
Let $Y$ be the Cantor cube $2^\cont$ with its usual topology. As in~\cite{JVM2017}, let $X=C_1\un C_2$ as a subspace of $Y$ where 
$$C_i=\{x\in Y:|\{\alpha<\cont:x(\alpha)|<\omega\}$$
for $i\in\{0,1\}$. Let $\scr{C}=\{C_1,C_2\}$ and note $\scr{C}$ is a cover of $X$ consisting of countably tight subspaces. It is easily seen that each $C_i$ is dense in $X$. Also, as mentioned in~\cite{JVM2017}, $t(X)=\cont=2^\omega$. Therefore, by Lemma~\ref{lemmaDense}, $wt(X)=\omega$.

Note that $X$ is in fact a $\sigma$-compact group, and thus is an example of a weakly countably tight homogeneous space that is not countably tight.
\end{example}

\section{A generalization using $G^c_\kappa$-sets.}

\begin{definition}\label{gk}
Let $X$ be a space and $\kappa$ be a cardinal. A $G_\kappa$-\emph{set} in $X$ is a $\kappa$-intersection of open sets of $X$. (``$G_\delta$-set'' is the common name for a $G_\omega$-set). A $G_\kappa^c$-\emph{set} is set $G\sse X$ such that there exists a $\kappa$-family of open sets $\scr{U}$ such that $G=\Meet\scr{U}=\Meet_{U\in\scr{U}}\closure{U}$. The $G_\kappa$-\emph{modification} $X_\kappa$ of $X$ is the space formed on the set $X$ with topology generated by the collection of $G_\kappa$-sets. The $G^c_\kappa$-\emph{modification} $X^c_\kappa$ of $X$ is the space formed on the set $X$ with topology generated by the collection of $G^c_\kappa$-sets.
\end{definition}

We work towards a generalization of Theorem~\ref{mt1} involving covers by $G^c_\kappa$ sets. We again develop a closing-off argument that, by use of $G^c_\kappa$ sets, adds additional complexity to the argument used in the proof of Theorem~\ref{mt1}. We begin with a series of propositions and lemmas. 

In the proof of Lemma 3.2 in \cite{JVM2017} the following proposition is used in the case $\kappa=\omega$. It is mentioned as having an $\omega$-length proof by induction. We give this proof below.

\begin{proposition}\label{saturated}
Let $X$ be a space, $wt(X)=\kappa$, and let $\scr{C}$ be a cover witnessing that $wt(X)=\kappa$. Then for all $x\in X$ there exists $S(x)\in[X]^{\leq 2^\kappa}$ such that $x\in S(x)$ and $S(x)$ is $\scr{C}$-saturated.
\end{proposition}

\begin{proof}
Fix $x\in X$. We inductively construct a sequence of subsets $A_n\sse X$ for all $n$ such that $|A_n|\leq 2^\kappa$. For all $C\in\scr{C}$ there exists $A_C\in[C]^{\leq 2^\kappa}$ such that $x\in\closure{A_C}$. Let $A_1=\Un_{C\in\scr{C}}A_C$ and note that $|A_1|\leq 2^\kappa$.  Suppose that $A_n$ has been constructed for $n<\omega$. For all $y\in A_n$ and for all $C\in\scr{C}$ there exists $A(C,y,n)\in [C]^{\leq 2^\kappa}$ such that $y\in\closure{A(C,y,n)}$. Let $A_{n+1}= \Un\{A(C,y,n):y\in A_n, C\in\scr{C}\}$ and note that $|A_{n+1}|\leq 2^\kappa$.

Define $S(x)=\{x\}\un\Un_{n<\omega}A_n$. Then $x\in S(x)$ and $|S(x)|\leq 2^\kappa$. We show that $S(x)$ is $\scr{C}$-saturated. Let $C\in\scr{C}$. We show $S(x)\sse\closure{C\meet S(x)}$. By above, $x\in\closure{A_C}$ and $A_C\sse C\meet A_1\sse C\meet S(x)$. Thus $x\in\closure{C\meet S(x)}$. Now let $n\leq\omega$ and $y\in A_n$. By above, $y\in\closure{A(C,y,n)}$ and $A(C,y,n)\sse C\meet A_{n+1}\sse C\meet S(x)$. Thus $y\in\closure{C\meet S(x)}$. This shows that $S(x)\sse\closure{C\meet S(x)}$ and that $S(x)$ is $\scr{C}$-saturated.
\end{proof}

The proof of Lemma~\ref{lemma3} below uses the notion of a $\theta$-network and the cardinal invariant $nw_\theta(X)$. These were introduced in \cite{CPR2012}. $nw_\theta(X)$ is a variation of the network weight $nw(X)$ that is typically useful if $X$ is not necessarily assumed to be regular. It is clear that $nw(X)=nw_\theta(X)$ if $X$ is regular.

\begin{definition}
A \emph{network} for a space $X$ is a family $\scr{N}$ of subsets of $X$ such that whenever $U$ is open in $X$ and $x\in U$ there exists $N\in\scr{N}$ such that $x\in N\sse U$. The \emph{network weight} $nw(X)$ of $X$ is the least cardinality of a network for $X$.  A $\theta$-\emph{network} of a space $X$ is a family $\scr{N}$ of subsets of $X$ such that if $U$ is open in $X$ and $x\in U$ there exists $N\in\scr{N}$ such that $x\in N\sse\closure{U}$. The $\theta$-\emph{network weight} $nw_\theta(X)$ is the least cardinality of a $\theta$-network for $X$.
\end{definition}

\begin{lemma}\label{lemma3}
Let $X$ be a space, $D\sse X$, $wt(X)\leq\kappa$, and let $\scr{G}$ be a cover of $\closure{D}$ consisting of $G^c_\kappa$-sets of $X$. Then there exists $\scr{G}^\prime\in[\scr{G}]^{\leq 2^\kappa}$ such that $\scr{G}^\prime$ covers $\closure{D}$.
\end{lemma}

\begin{proof}
Let $\scr{C}$ be a cover of $X$ witnessing that $wt(X)\leq\kappa$. By Propostion~\ref{saturated}, for all $d\in D$ there exists $S(d)\in[X]^{\leq 2^\kappa}$ such that $d\in S(d)$ and $S(d)$ is $\scr{C}$-saturated. Let $S=\Un_{d\in D}S(d)$ and note that $S$ is $\scr{C}$-saturated and $|S|\leq |D|\cdot 2^\kappa$. 

We now show $nw_\theta(\closure{S})\leq |D|^\kappa$. Let $\scr{N}=\{\closure{A}:A\in[S]^{\leq\kappa}\}$. We show $\scr{N}$ is a $\theta$-network for $\closure{S}$. Suppose $x\in U\meet\closure{S}$ where $U$ is an open set in $X$. As $\scr{C}$ is a cover of $X$, there exists $C\in\scr{C}$ such that $x\in C$. As $S$ is $\scr{C}$-saturated, we have that $C\meet S$ is dense in $S$ and that $x\in\closure{U\meet S\meet C}$. As $t(C)\leq\kappa$ there exists $A\sse [U\meet S\meet C]^{\leq\kappa}$ such that $x\in\closure{A}\sse cl(U\meet\closure{S})\meet\closure{S}=cl_{\closure{S}}(U\meet\closure{S})$.
This verifies $\scr{N}$ is a $\theta$-network for $\closure{S}$. Thus, $nw_\theta(\closure{S})\leq|S|^{\leq\kappa}\leq\left(|D|\cdot 2^\kappa\right)^\kappa=|D|^\kappa$.

Now, by Lemma 3.4 in \cite{CPR2012}, it follows that $nw(\closure{S}^c_\kappa)\leq nw_\theta(\closure{S})^\kappa\leq\left(|D|^\kappa\right)^\kappa=|D|^\kappa$. Let $\scr{M}$ be a network for $\closure{S}^c_\kappa$ such that $|\scr{M}|\leq |D|^\kappa$. Let $x\in\closure{D}$. There exists $G_x\in\scr{G}$ such that $x\in G_x$. Let $\scr{U}_x$ be a family of open sets in $X$ such that $G_x=\Meet\scr{U}_x=\Meet_{U\in\scr{U}_x}\closure{U}$. Now, 
\begin{align}
G_x\meet\closure{S}&=\Meet_{U\in\scr{U}_x}(U\meet\closure{S})\sse\Meet_{U\in\scr{U}_x}cl_{\closure{S}}(U\meet\closure{S})\notag\\
&=\Meet_{U\in\scr{U}_x}cl(U\meet\closure{S})\meet\closure{S}\sse\Meet_{U\in\scr{U}_x}\closure{U}\meet\closure{S}=G_x\meet\closure{S}.\notag
\end{align}
Thus, $G_x\meet\closure{S}=\Meet_{U\in\scr{U}_x}(U\meet\closure{S})=\Meet_{U\in\scr{U}_x}cl_{\closure{S}}(U\meet\closure{S})$ is a $G^c_\kappa$-set in $\closure{S}$ and $x\in G_x\meet\closure{S}$. There exists $M_x\in\scr{M}$ such that $x\in M_x\sse G_x\meet\closure{S}$. 

Set $\scr{M}^\prime=\{N_x:x\in\closure{D}\}$ and note $|\scr{M}^\prime|\leq |\scr{M}|\leq |D|^\kappa$. For all $M\in\scr{M}^\prime$ there exists $G_M\in\scr{G}$ such that $M\sse G_M$. Then, if $x\in\closure{D}$, we have $N_x=M$ for some $M\in\scr{M^\prime}$ and $x\in N_x=M\sse G_M$. This shows $\scr{G}^\prime=\{G_M:M\in\scr{M}^\prime\}$ covers $\closure{D}$ and $|\scr{G}^\prime|\leq |\scr{M}^\prime|\leq |D|^\kappa$.
\end{proof}

Observe now that if $X$ is Hausdorff then Theorem~\ref{cardClosure} follows from Lemma~\ref{lemma3} after setting $\kappa=wt(X)\psi_c(X)$ and $\scr{G}=\{\{x\}:x\in\closure{D}\}$. We prove now the central theorem in this paper from which all important corollaries follow.

\begin{mt}\label{main}
For any space $X$ and cardinal $\kappa$, $L(X^c_\kappa)\leq 2^{L(X)wt(X)\cdot\kappa}$.
\end{mt}

\begin{proof}
Without loss of generality we can assume $L(X)wt(X)\leq\kappa$. Let $\scr{C}$ be a cover of $X$ witnessing that $wt(X)\leq\kappa$. By Proposition~\ref{saturated} for every $x\in X$ we can fix a $\scr{C}$-saturated set $S(x)\in[X]^{\leq 2^\kappa}$ with $x\in S(x)$.

Fix a cover $\scr{G}$ of $X$ consisting of $G^c_\kappa$-sets of $X$. For all $G\in\scr{G}$ fix a family $\scr{U}(G)$ of open sets such that $G=\Meet\scr{U}(G)=\Meet_{U\in\scr{U}(G)}\closure{U}$. If $\scr{G}^\prime\sse\scr{G}$, let $\scr{U}(\scr{G}^\prime)=\Un\{\scr{U}(G):G\in\scr{G}^\prime\}$. 
We build an increasing sequence $\{D_\alpha:\alpha<\omega_1\}$ of subsets of $X$ and an increasing chain $\{\scr{G}_\alpha:\alpha<\omega_1\}$ of subsets of $\scr{G}$ such that 
\begin{enumerate}
\item $|\scr{G}_\alpha|\leq 2^\kappa$ and $|D_\alpha|\leq 2^\kappa$.
\item $\closure{D_\alpha}\sse\Un\scr{G}_\alpha$ and each $D_\alpha$ is $\scr{C}$-saturated
\item if $\scr{V}\in[\scr{U}(G_\alpha)]^{\leq\kappa}$ is such that $X\minus\Un\scr{V}\neq\es$, then $D_{\alpha+1}\minus\Un\scr{V}\neq\es$.
\item $D_\alpha$ is $\scr{C}$-saturated for all $\alpha<\kappa^+$.
\end{enumerate}
Let $D_0$ be an arbitrary $\scr{C}$-saturated set of size $\cont$. For limit ordinals $\beta<\kappa^+$, let $D_\beta=\Un_{\alpha<\beta}D_\alpha$. Then $|D_\beta|\leq 2^\kappa$. Apply Lemma~\ref{lemma3} to $\closure{D_\beta}$ to obtain $\scr{G}_\beta$ as required. As the union of $\scr{C}$-saturated sets is $\scr{C}$-saturated, as noted in \cite{JVM2017}, it follows that $D_\beta$ is $\scr{C}$-saturated. For successor ordinals $\beta+1$, for all $\scr{V}\in[\scr{U}(\scr{G}_\beta)]^{\leq\kappa}$ with $X\minus\Un\scr{V}\neq\es$, let $x_\scr{V}\in X\minus\Un\scr{V}$. Define
$$D_{\beta+1}=D_\beta\un\Un\{S(x_\scr{V}):\scr{V}\in[\scr{U}(\scr{G}_\beta)]^{\leq\kappa}\textup{ and }X\minus\Un\scr{V}\neq\es\}.$$
As $|S(x_\scr{V})|\leq 2^\kappa$ and $\left|[\scr{U}(\scr{G}_\beta)]^{\leq\kappa}\right|\leq 2^\kappa$, we have $|D_{\beta+1}|\leq 2^\kappa$. As each $S(x_\scr{V})$ is $\scr{C}$-saturated, we have that $D_{\beta+1}$ is $\scr{C}$-saturated. Apply Lemma~\ref{lemma3} again to $\closure{D_{\beta+1}}$ to obtain $\scr{G}_{\beta+1}$. 

Let $D=\Un\{D_\alpha:\alpha<\omega_1\}$. By Lemma~\ref{lemma2} it follows that $\closure{D}=\Un\{\closure{D_\alpha}:\alpha<\kappa^+\}$. Let $\scr{G}^\prime=\Un\{\scr{G}_\alpha:\alpha<\kappa^+\}$. We show that $\scr{G}^\prime$ is a cover of $X$. Suppose there exists $x\in X\minus\Un\scr{G}^\prime$. Then $x\notin G$ for every $G\in\scr{G}^\prime$. Thus for all $G\in\scr{G}^\prime$ there exists $U(G)\in\scr{U}(G)$ such that $x\notin\closure{U(G)}$. Now $\scr{U}=\{U(G):G\in\scr{G}^\prime\}$ covers $\Un\scr{G}^\prime$ and hence covers $\closure{D}$. As $L(X)\leq\kappa$ and $L(X)$ is hereditary on closed sets, there exists $\scr{V}\in[\scr{U}]^{\leq\kappa}$ such that $\closure{D}\sse\Un\scr{V}$. Now $\scr{V}\in[\scr{U}(\scr{G}_\alpha)]^{\leq\kappa}$ for some $\alpha<\kappa^+$. Since $x\in X\minus\Un\scr{V}$, by (3) above $D_{\alpha+1}\minus\Un\scr{V}\neq\es$. Then
$$\es\neq D_{\alpha+1}\minus\Un\scr{V}\sse D\sse\closure{D}\sse \Un\scr{V},$$
a contradition. Therefore $X=\Un\scr{G}^\prime$ and $L(X^c_\delta)\leq 2^\kappa$.
\end{proof}

A corollary to the above theorem is a strengthening of Pytkeev's theorem (Theorem~\ref{pytkeev}).

\begin{corollary}
If $X$ is compact and $\kappa$ is a cardinal, then $|X|\leq 2^{wt(X)\cdot\kappa}$.
\end{corollary}

We conclude this section by showing that Theorem~\ref{mt1} is in fact another corollary of the Main Theorem.

\begin{corollary}
If $X$ is Hausdorff then $|X|\leq 2^{L(X)wt(X)\psi(X)}$.
\end{corollary}

\begin{proof}
Let $\kappa=L(X)wt(X)\psi_c(X)$ and note that $\kappa=L(X)wt(X)\psi(X)$ as $\psi_c(X)\leq L(X)\psi(X)$. Observe that $X^c_\kappa$ is discrete, and thus by the Main Theorem, $|X|=|X^c_\kappa|=L(X^c_\kappa)\leq 2^\kappa$.
\end{proof}

\section{Power homogeneous compacta.}

In this section we apply our Main Theorem to the context of power homogeneous compacta. Theorem~\ref{ph} below is a generalization of the~\juhasz- van Mill theorem (Theorem~\ref{jvm}) and extends that result into the power homogeneous setting. This further demonstrates that cardinality bounds for general spaces have companion bounds for homogeneous and power homogeneous spaces, both derived from a single result involving the $G_\kappa$-modification or variations.

In Theorem~\ref{jvm} the space $X^\omega$ is required to be homogeneous. In the proof of that result the countability of the power leads to the crucial fact that $\pi_\chi(X^\omega)=\pi_\chi(X)=\omega$. When generalizing to the power homogeneous case where $X^\kappa$ is homogeneous for some cardinal $\kappa$, we need a way to still guarantee that $\pi_\chi(X)=\omega$. The following lemma of Ridderbos is used in this connection. Lemma~\ref{lemma4} shows that the existence of a dense set of points of small $\pi$-character in a power homogeneous space guarantees that every point in the space has small $\pi$-character. 

\begin{lemma}[Ridderbos, \cite{Rid2009}]\label{lemma4}
Let $X$ be a power homogeneous space and let $\kappa$ be an infinite cardinal. If the set $D=\{x\in X:\pi_\chi(x,X)\leq\kappa\}$ is dense in $X$, then $\pi_\chi(X)\leq\kappa$.
\end{lemma}

We now show that the $\pi$-character of a power homogeneous compactum can still be controlled if each closed subspace has a point with small $\pi$-character.

\begin{lemma}\label{lemma5}
Let $X$ be a power homogeneous compactum and let $\kappa$ be an infinite cardinal. If for all closed subsets $F\sse X$ there exists $x\in F$ such that $\pi_\chi(x,F)\leq\kappa$, then $\pi_\chi(X)\leq\kappa$.
\end{lemma}

\begin{proof}
We show first that $D=\{x\in X:\pi_\chi(x,X)\leq\kappa\}$ is dense in $X$. Let $U$ be a non-empty open set. As $X$ is normal there exists a non-empty closed (and compact) $G_\delta$-set such that $G\sse U$. (See 2.4.8 in \cite{RidThesis}, for example). As $X$ is compact, we have $\chi(G,X)=\psi(G,X)=\omega$. As $G$ is closed, by assumption there exists $x\in G$ such that $\pi_\chi(x,G)\leq\kappa$. Therefore, $\pi_\chi(x,X)\leq\pi_\chi(x,G)\chi(G,X)\leq\kappa\cdot\omega=\kappa$. (See 2.4.3 in \cite{RidThesis}). As $x\in G\sse U$, we see that $D$ is dense in $X$. Now apply Lemma~\ref{lemma4}.
\end{proof}

We apply the Main Theorem and the following three lemmas to establish Theorem~\ref{ph}, the main result in this section. The deep result in Lemma~\ref{lemma6} plays a key role in~\cite{JVM2017} and will here as well. It represents a strengthening of 2.2.4 in \cite{Arh1978} in the case where $X$ is countably tight. A subspace $A$ of a space $X$ is \emph{subseparable} if there exists a countable set $C\sse X$ such that $A\sse\closure{C}$. (Note $C$ need not be a subset of $A$). 

\begin{lemma}[\juhasz, van Mill \cite{JVM2017})]\label{lemma6}
Every $\sigma$-CT compactum $X$ has a non-empty subseparable $G_\delta$-set.
\end{lemma}

Lemma~\ref{lemma7} is a modified version of Corollary 2.9 in~\cite{AVR2007}

\begin{lemma}\label{lemma7}
Let $X$ be a power homogeneous space such that $\pichar{X}\leq\omega$. Suppose there exists a non-empty subseparable closed $G^c_\delta$-set in $X$. Then every point of $X$ is contained in a subseparable closed $G^c_\delta$-set.
\end{lemma}

Lemma~\ref{lemma8} was shown by Ridderbos in~\cite{Rid2006}. It was in fact generalized in~\cite{BC2017}: if $X$ is power homogeneous, $D\sse X$, and $U$ is an open set such that $U\sse\closure{D}$, then $|U|\leq |D|^{\pi_\chi(X)}$.

\begin{lemma}[Ridderbos \cite{Rid2006}]\label{lemma8}
If $X$ is power homogeneous and Hausdorff then $|X|\leq d(X)^{\chi(X)}$.
\end{lemma}

\begin{theorem}\label{ph}
Let $X$ be a power homogeneous compactum and suppose there exists a countable cover $\scr{C}$ of $X$ such that each $C\in\scr{C}$ is CT and dense in $X$. Then $|X|\leq\cont$
\end{theorem}

\begin{proof}
As $X$ is $\sigma$-CT, by Lemma 2.4 in \cite{JVM2017} every non-empty closed subspace $F$ of $X$ has a point $x\in F$ such that $\pi_\chi(x,F)\leq\omega$. As $X$ is power homogeneous, by Lemma~\ref{lemma5}, $\pi_\chi(X)\leq\omega$. Since each $C\in\scr{C}$ is countably tight and dense in $X$, $|\scr{C}|\leq\cont$, and $\pi_\chi(X)\leq\cont$, it follows by Lemma~\ref{lemmaDense} that $wt(X)\leq\omega$.

By Lemma~\ref{lemma6}, there exists a non-empty subseparable closed $G_\delta$-set in $X$. This set is in fact a closed $G^c_\kappa$-set as $X$ is compact. By Lemma~\ref{lemma7}, there exists a cover $\scr{G}$ of $X$ consisting of subseparable $G^c_\delta$-sets. Applying the Main Theorem, there exists $\scr{G}^\prime\in[\scr{G}]^{\leq\cont}$ such that $X=\Un\scr{G}^\prime$. For all $G\in\scr{G}^\prime$ there exists $H_G\in[X]^{\leq\omega}$ such that $G\sse\closure{H_G}$. Therefore, $X=\Un\scr{G}^\prime\sse\Un_{G\in\scr{G}^\prime}\closure{H_G}\sse\closure{\Un_{G\in\scr{G}^\prime}H_G}$ and $H=\Un_{G\in\scr{G}^\prime}H_G$ is dense in $X$. Since $|H|\leq\cont$, we see that $d(X)\leq\cont$. Finally, applying Lemma~\ref{lemma8}, we see that $|X|\leq d(X)^{\pi_\chi(X)}\leq\cont^\omega=\cont$.
\end{proof}

Finally, we show that a theorem of Bella and Spadaro from \cite{BS2017} follows from the Main Theorem and also generalizes to the power homogeneous setting. 

\begin{theorem}\label{phbs}
Let $X$ be a power homogeneous compactum which is the union of countably many CT dense subspaces. Then $L(X_\delta)\leq\cont$.
\end{theorem}

\begin{proof}
As in the first paragraph in the proof above, we see that $wt(X)\leq\omega$. By the Main Theorem, $L(X^c_\delta)\leq\cont$. As $X$ is regular, $X_\delta\approx X^c_\delta$ and $L(X_\delta)\leq\cont$. 
\end{proof}

We conclude with two questions. First, it was shown by~\sapirovskii~that if $X$ is compact then $\pi_\chi(X)\leq t(X)$. As $wt(X)\leq t(X)$, it is natural to ask the following.

\begin{question}
If $X$ is a compactum, is $\pi_\chi(X)\leq wt(X)$?
\end{question}

Furthermore, in light of Theorem~\ref{ph} and noting that a weakly countably tight space may not have a \emph{countable} cover $\scr{C}$ that witnesses $wt(X)\leq\kappa$, it is natural to ask the following.
\begin{question}
If $X$ is a weakly countably tight power homogeneous compactum, is $|X|\leq\cont$?
\end{question}

\end{document}